\documentclass[10pt]{amsart}
\newcommand{\bt}{\begin{Theorem}}
\newcommand{\et}{\end{Theorem}}
\newcommand{\bi}{\begin{itemize}}
\newcommand{\ei}{\end{itemize}}
\newcommand{\bea}{\begin{eqnarray}}
\newcommand{\ba}{\begin{array}}
\newcommand{\eea}{\end{eqnarray}}
\newcommand{\ea}{\end{array}}

\newcommand{\lgra}{\longrightarrow}
\newcommand{\w}{\omega}

\newtheorem{Definition}{Definition}[section]
\newtheorem{Theorem}[Definition]{Theorem}
\newtheorem{Lemma}[Definition]{Lemma}

\newtheorem{Corollary}[Definition]{Corollary}
\newtheorem{Example}[Definition]{Example}
\newtheorem{Remark}[Definition]{Remark}

\newcommand{\be}{\begin{equation}}
\newcommand{\ee}{\end{equation}}
\newcommand{\bvb}{\begin{verbatim}}

\newcommand{\ds}{\displaystyle}

\newcommand{\evb}{\end{verbatim}}

\newcommand{\E}{\mathbb E}

\newcommand{\R}{\mathbb R}%
\newcommand{\C}{\mathbb C}%
\renewcommand{\Re}{\mbox{Re }}
\begin{document}
\baselineskip14pt
\author[A. Kumar]{Ashisha Kumar}
\address[Ashisha Kumar]{Department of mathematics, Indian
Institute of Science, Bangalore 560012, Karnatka, India, E-mail:
ashisha10@math.iisc.ernet.in}

\author[Ray, S. K.]{Swagato K. Ray}
\address[Swagato K. Ray]{Stat-math Unit, Indian Statistical Institute, 203 B. T. Road, Kolkata 700108,
India, E-mail:skray@iitk.ac.in}

\title[End point estimates for Radon transform]{End point estimates for Radon transform of radial functions on Non-Euclidean spaces}

\begin{abstract} We prove end point estimate for Radon transform of radial functions on affine Grasamannian and real hyperbolic space.
 We also discuss analogs of these results on the sphere.
 \end{abstract}
\maketitle

\section{Introduction}
Given a function $f$ on $\R^n$, $n\geq 2$, the Radon transform of
$f$ is defined by the formula
\be
 Rf(\w,t)=\int_{\R^{n-1}}f(t\w+y^{\prime}) \ d
y^{\prime},~\quad  \w \in S^{n-1},~ t \in \R, \ee whenever the above
integral makes sense. One can analogously define the $d\mbox{-}$plane
transform of a function as the integral of the function over
$d\mbox{-}$dimensional planes with respect to the
$d\mbox{-}$dimensional Lebesgue measure. Precisely, given $1 \leq d
\leq n-1$ and a $d\mbox{-}$dimensional linear subspace $\Pi$ of
$\R^n,$ one defines
\begin{equation}
T_df(x,\Pi)=\int_{\Pi}f(x-y)\ d\lambda_{d}(y),\label{dplane definition} \end{equation} where
$\lambda_{d}$ denotes the $d\mbox{-}$dimensional Lebesgue measure on
$\Pi$. By a simple application of Fubini's theorem one can see that
the above  integrals converge absolutely for $f\in L^p(\R^n)$, if $p=1$. But for $p>1$ the
situation is rather involved. It was proved in \cite{S} that if
$f\in L^p(\R^n)$, $1\leq p<\frac{n}{d}$, then the lower dimensional
integral (\ref{dplane definition}) is well defined for allmost every $x$ and $\Pi$.
Roughly speaking, the above phenomenon occurs because the $d\mbox{-}$plane
transform is closely related to the Fourier transform and the Reisz potential. A natural question then
 is to ask: If $f\in L^p(\R^n)$, $1 \leq p<\frac{n}{d}$, then what
 can one say about the size of $T_df?$ Of course the size has to be
 measured with respect to the natural measure on the set of $d\mbox{-}$dimensional affine subspaces.
 The above problem was completely settled for $d=n-1$ in \cite{OS}.
It was shown in \cite{OS} that there exists a positive constant $C$
such that for all $f \in C_c^{\infty}(\R^n )$ the following
inequality holds,
\be\left(\int_{S^{n-1}}\left[\int_{\R}|Rf(\w,t)|^qdt
\right]^{p^{\prime}/q} d\sigma_{n-1}(\w) \right)^{1/p^{\prime}} \leq
C \|f\|_{L^p(\R^n)}~,\label{ossynp} \ee $1 \leq p < n/(n-1)$, $
1/q=(n/p)-n+1$, $1/p'=1-(1/p)$. Here $\sigma_{n-1}$ denotes the
normalized rotation invariant measure on the sphere.

Apart from the above result  another intriguing observation related
to the end-point estimate of (\ref{ossynp}) involving Lorentz norm was made in
\cite{OS}. It was shown in \cite{OS} that the end-point estimate of
(\ref{ossynp}) holds for $n\geq 3,$ that is, \be
\left(\int_{S^{n-1}} (\mbox{sup}_{t\in\R}~|Rf(\w,t)|)^n
 \ d\sigma_{n-1}(\w)\right)^{1/n} \leq C_n \|f\|_{L^{\frac{n}{n-1},1}(\R^n)}
.\ee But the above estimate fails for the case $n=2$. A consequence
of the above is that if $f\in L^{\frac{n}{n-1},1}(\R^n), n\geq 3$ then its Radon
transform is well defined for almost every hyperplane. The failure
of the end-point estimate for $n=2$ can be attributed to the
existence of compact Kakeya sets in $\R^2$ of arbitrary small
Lebesgue measure (\cite{OS}, p. $642$). Since the indicator function
of a set with radial symmetry and arbitrarily small
$L^{2,1}$ norm cannot contain line segments of a fixed length in every direction, one can still hope for an end-point estimate for the
radial functions in the case $n=2$. This viewpoint was adopted in
\cite{DNO} and the question regarding the end-point estimate for the
$d\mbox{-}$plane transform of radial functions was answered: if $1 \leq d \leq n-1$ then for all radial
functions $f$ the following estimate holds,
\begin{equation} \mbox{sup}_{(x,\Pi)\in \R^n\times G_{n,d}}|T_df(x,\Pi)|\leq
C\|f\|_{L^{\frac{n}{d},1}(\R^n)}.\label{DNOsynp}
\end{equation}
Here $G_{n,d}$ stands for the set of $d\mbox{-}$dimensional linear
subspaces of $\R^n$ and $C>0$ is a constant independent of $f$. In
this paper our main objective is to establish analogs of
(\ref{DNOsynp}) on affine Grassmannian, real hyperbolic space and
the sphere.

There is a wealth of information available in the literature on the study of $d\mbox{-}$plane transform on affine Grassmannian
 (see for instance \cite{Go,GK1,GK2,He1,Str1}). However, the study of $d\mbox{-}$plane transform of $L^p$ functions on
affine Grassmannian seems to be relatively new.
Investigation regarding $L^p-L^q$ mapping property of the $d$-plane
transform for affine Grassmannian was recently initiated in \cite{R5}.
This notion of $d$-plane transform generalizes the notion of $d$-plane transform of functions
defined earlier. This motivated us to look for an estimate similar to (\ref{DNOsynp})
for the $d\mbox{-}$plane transform of radial functions on affine Grassmannian. Our main result in section $2$ shows that it is possible to
prove an appropriate analogue of (\ref{DNOsynp}) for generalized $d$-plane
transform of radial functions on affine Grassmannian (see Theorem \ref{grast}).

Next we turn towards the real hyperbolic space. There are several
papers which deal with $L^p-L^q$ mapping property
of the totally geodesic radon transform on real hyperbolic spaces
(see for instance \cite{BR1, BR2, Is, Str2} and references there
in). But none of these papers address the question of end point
estimates for the Radon transform. The best possible $L^p-L^q$
mapping property of the Radon transform (or more generally the
$d$-plane transform) on these spaces seem to be a hard problem.
However, it turns out that the same problem on the restricted class
of radial functions is not very difficult to tackle. In fact, analogues of (\ref{DNOsynp})
has already appeared in the non Euclidean setup. The first one is by
Cowling, Meda and Setti \cite{CMS}. While working on Kunze-Stein
phenomena on homogeneous trees they proved that the horospherical
Radon transform of radial functions defines a continuous operator
from $L^{2,1}$ to $L^{\infty}$ (\cite{CMS}, Theorem $2.5$). The
second result is by Ionescu \cite{Io}. It was shown in \cite{Io}
that for rank one Riemannian symmetric spaces of non compact type
the horospherical Radon transform of radial functions is also continuous
from $L^{2,1}$ to $L^{\infty}$ (\cite{Io}, Proposition $2$).

These results motivated us to consider the $d\mbox{-}$dimensional
totally geodesic Radon transform of radial functions on real
hyperbolic space $\mathbb H^n$ and the sphere $S^n$. A result in
\cite{BR1} says that: If $f\in L^p(\mathbb H^n)$, $1\leq p
<(n-1)/(d-1)$, then $f$ is integrable over almost every
$d\mbox{-}$dimensional totally geodesic submanifold. In analogy with
(\ref{DNOsynp}) it is natural to enquire about the validity of
the end-point estimate only for radial functions. We
answer the question in Section $3$ as follows (see Theorem
\ref{endabel}): the $d$-plane transform restricted to the class of radial functions defines a
continuous linear map from $L^{\frac{n-1}{d-1},1}(\mathbb H^n)$ to
$L^{\infty}(\R^+)$ if $n\geq 3$. As a consequence it follows that if $f$
is radial and $f \in L^{\frac{n-1}{d-1},1}(\mathbb H^n)$, $n\geq 3$, then $f$ is
integrable over almost every $d\mbox{-}$dimensional totally geodesic
submanifold. This is in the same spirit as $\R^n$, homogeneous trees
and rank one symmetric spaces of non compact type. However there are some non Euclidean
consequences of the above result (see Corollary \ref{coroh}). One such is that if $f \in
L^{\frac{n-1}{d-1},1}(\mathbb H^n)$ and is radial then its Radon
transform has an exponential decay at infinity. As a consequence we will see
that the $d\mbox{-}$dimensional totally geodesic Radon
transform of radial functions is also continuous from
$L^{\frac{n-1}{d-1},1}(\mathbb H^n)$ to $L^{n-1,\infty}(\R^+).$ This is
in sharp contrast with the Euclidean spaces.

In Section $4$ we consider the case of the sphere. The situation here is
very different because of compactness. It is known from \cite{R3}
that in this case the $d$-plane transform is continuous from $L^p(S^n)$ to
$L^p(SO(n+1)/SO(n-d)\times SO(d+1)),$ $1\leq p\leq \infty$ (see also
\cite{Str2}). Here the quotient space $SO(n+1)/SO(n-d)\times
SO(d+1)$ is viewed as the space of $d\mbox{-}$dimensional totally
geodesic submanifolds of $S^n.$ In this case one can show that the
exact analogue of (\ref{DNOsynp}), for functions which are invariant
under the action of $SO(n),$ is not true (Example \ref{examsphere}).
It turns out that one can prove a result analogous to
(\ref{DNOsynp}) if the $SO(n+1)$ invariant measure on
$SO(n+1)/SO(n-d)\times SO(d+1)$ is considered along with a weight
which is naturally associated with the structure of the set of
$d\mbox{-}$dimensional totally geodesic submanifolds (see Theorem
\ref{endabelsph}).

\section[Affine Grassmannian]{Affine Grassmannian}
\subsection[Notation and Preliminaries]{Notation and Preliminaries}
Let $G_{n,k}$ be the standard Grassmann manifold of all
$k\mbox{-}$dimensional linear subspaces of $\R^n,~0\leq k <n.$
The rotation group $SO(n)$ acts on $G_{n,k}$ by $A\cdot \xi=A(\xi), A\in SO(n), \xi \in G_{n,k}.$
The above action can easily be seen to be transitive and moreover the isotropy subgroup at
$\xi_{0}=\mbox{span }\{e_1,\ldots, e_{k}\}$ is $SO(k)\times SO(n-k)$ (see \cite{M}, p. 140).
Consequently we can identify $G_{n,k}$ with the compact homogeneous space $SO(n)/SO(k)\times SO(n-k).$
This identification allows us to talk about the $SO(n)$ invariant measure on $G_{n,k}$ with total mass $1$.
This measure will be denoted by $d\xi$.
Let $G(n,k)$ denotes the set of all $k\mbox{-}$dimensional
affine subspaces of $\R^n.$ Given $\xi \in G_{n,k}$ and $u \in \xi^{\perp}$ the translated plane $\xi+u \in G(n,k).$
Moreover, given any $\tau \in G(n,k)$ there exists unique $\xi \in G_{n,k}$ and unique $u\in \xi^{\perp}$ such that $\tau=\xi+u.$ Consequently,
$G(n,k)$ can be parameterized by the pair $(\xi,u)$, $\xi \in G_{n,k}$, $u\in \R^{n-k}$.
 The manifold $G(n,k)$ can also be viewed as a homogeneous space of the Euclidean
 motion group $M(n)=O(n)\times \R^n$. Given an element $(T,v) \in M(n)$ and $\tau \in G_(n,k)$ we define
 $$(T,v)\cdot \tau=\left\{\left(\begin{matrix} T & v\\
                                        0 & 1\end{matrix}\right)\left(\begin{matrix}w\\1\end{matrix}\right)| w\in \tau \right\}$$
 This action is known to be transitive. Let $\{e_1,e_2,\ldots,e_n\}$ be the standard orthonormal basis of $\R^n$
 and $\tau_k=\Re_1\oplus\ldots\oplus \Re_k$ (we assume $\tau_k=0$ if $k=0$. It then follows that the
 isotropy subgroup $H_k$ at $\tau_k$ is given by
 \bea H_k&=&\{(T,v)\in M(n) : (T,v)\cdot \tau_k=\tau_k\}\nonumber\\
 &=&\left\{\left(\begin{matrix} T1 & 0 & v\\
                                0 & T_2 & 0\\
                                0 & 0  & 1\end{matrix}\right):T_1 \in O(k), T_2\in O(n-k), v\in \tau_k \right\}.\eea
 Hence $G(n,k)\approx M(n)/(O(k)\times O(n-k))\times \R^k. $

 Let $\xi$ be a $d\mbox{-}$plane, $k<d<n$, and $\tau$ a $k\mbox{-}$plane with $\tau \subset \zeta.$ According to our convention we have
 $\zeta=(\eta,v),$ $\tau=(\xi,u)$ where $\eta \in G_{n,d},\ \xi \in G_{n,k},\ v \in \eta^{\perp}$ and $ u \in \zeta^{\perp}.$
 Since $\tau \subset \xi$ it then follows immediately that $\xi \subset \eta.$
 We claim that there exists $x\in \xi^{\perp}\cap\eta$
such that $u=v+x.$ Let $A\in \tau$ then $A=Z+u=W+v$, $ Z \in \xi$,
$W\in \eta$, $u\in \xi^{\perp}$, $v\in \eta^{\perp}$. Hence $u-v=W-Z\in \eta$.
On the other hand $\eta^{\perp}\subset \xi^{\perp}$ (as $\xi \subset \eta$) and hence $u-v \in \xi^{\perp}$.
If we choose $x=u-v$ then we are through.
We consider the product measure $d\mu_k(\tau)=d\xi du$ on $G(n,k)$,
where $d\xi$ is the measure on $G_{n,k}$ and $du$ is the $(n-k)\mbox{-}$dimensional
Lebesgue measure on $\R^{n-k}$.
For $n>d>k$ and $\eta \in G_{n,d}$, we denote $G_k(\eta)$ the Grassmann manifold of
$k\mbox{-}$dimensional linear subspaces of $\eta$. We write
$$\tau=(\xi,u)\in G(n,k), \quad u\in \xi^{\perp}; \quad \zeta=(\eta,v)\in G(n,d), \quad v\in \eta^{\perp}.$$
The $d\mbox{-}$plan transform of a function $f$ on $G(n,k)$ is a
function $R_df$ on $G(n,d)$ defined by

\be R_d f(\zeta)=\int_{\tau\subset\zeta} f(\tau)=\int_{\xi\subset
\eta}d_{\eta}\xi \int_{\xi^{\perp}\cap\eta} f(\xi,v+x) dx.\label{rg}\ee
Here $d_{\eta}\xi $ denotes the normalized measure on Grassmannian
$G_k(\eta)$ of all $k\mbox{-}$dimensional linear subspaces of
$\eta$ and $dx$ denotes the $(d-k)\mbox{-}$dimensional Lebesgue measure.

If $\tau=(\xi,u)\in G(n,k)$, the distance $\|u\|$ of the plane $\tau$ from the origin will be denoted by $|\tau|$.
A function $f$ defined on $G(n,k)$ is called radial if $f(\xi,u)=f(\xi',u')$ whenever $\|u\|=\|u'\|$.
This means that a radial functions depends only on the distance of the
$k\mbox{-}$plane from the origin. For radial functions the formula (\ref{rg}) can be written in a more concrete
form. Since $v\in \eta^{\perp}$ and $x\in \eta$ it follows from (\ref{rg}) that if $f$ is radial then
\be R_d f(\zeta)= R_d f(\eta,v)=C_{n,d,k}\int_{\R^{(d-k)}} f((\|v\|^2+\|x\|^2)^{1/2}) dx. \label{rrg}
\ee
This implies that $R_df$ is also a radial function on $G(n,d)$. By writing $\|v\|=s$ and using polar coordinates
on $\R^{d-k}$ we get the expression
 \be R_d f(\eta,v)=A_df(s)=C_{n,d,k}\int_{s}^{\infty} f(t) (t^2-s^2)^{\frac{d-k}{2}-1} t dt. \label{drg}
\ee
This formula will play a crucial role in the main result of this section.

We will need the notion of Lorentz spaces in the present as well as in the subsequent sections. We briefly recall some relevant results on Lorentz
spaces (see \cite{Graf, SW3} for details). Let $(M, m)$ be a
$\sigma$-finite measure space, $f:M\lgra \C$ be a measurable
function and $p\in [1, \infty)$, $q\in [1, \infty]$. We define
\begin{equation*}\|f\|^*_{p,q}=\begin{cases}\left(\frac qp\int_0^\infty [f^*(t)t^{1/p}]^q\frac{dt}t\right)^{1/q}
\ \textup{ when } q<\infty\\ \\ \sup_{t>0}td_f(t)^{1/p}\ \ \ \ \ \ \ \ \ \ \ \ \textup{ when }
q=\infty.\end{cases}\end{equation*}
Here $d_f$ is the distribution function of
$f$ and $f^*(t)=\inf\{s\mid d_f(s)\le t\}$ is the {\em
nonincreasing rearrangement} of $f$ (\cite[p.~45]{Graf}). We take
$L^{p,q}(M)$ to be the set of all measurable $f:M\lgra \C$ such
that $\|f\|^*_{p,q}<\infty$. By $L^{\infty, \infty}(M)$ and
$\|\cdot\|_{\infty, \infty}$ we mean respectively the space
$L^\infty(M)$ and the norm $\|\cdot\|_\infty$.
For $p, q\in [1, \infty)$ the following identity  gives an
alternative expression of $\|\cdot\|_{p,q}^*$ which we will use.
\begin{equation*}\frac qp\int_0^\infty (t^{1/p}f^*(t))^q\frac
{dt}t=q\int_0^\infty (td_f(t)^{1/p})^q\frac{dt} t.\end{equation*}
The proof of this identity can be found, for instance, in \cite{RS}.
For $p,q$ in the range above,
$L^{p,p}(M)=L^p(M)$ and if $q_1\le q_2$ then $\|f\|^*_{p, q_2}\le
\|f\|^*_{p, q_1}$ and consequently $L^{p, q_1}(M)\subset L^{p,
q_2}(M)$. It is easy to see from above that for the indicator function of a measurable set of finite measure $E$ we have $\|\chi_E\|_{L^{p,1}(M)}=m(E)^{1/p}$.

In this section and everywhere else we will follow the standard practice of using the letter $C$ for constant,
whose value may change from one line to another. Occasionally the
constant $C$ will be suffixed to show its dependency on important
parameters. We will also use the symbol $f(x)\asymp g(x)$ to mean that there exist two positive constants $C_1$, $C_2$ such that
$C_1f(x)\leq g(x)\leq c_2f(x)$ for appropriate values of $x$.

\subsection[Main result]{Main result}
Regarding the estimate of $R_df$ for $L^p$ functions on $G(n,k)$ it is known
from \cite{R5},{ Corallary 2.6} that if $f\in L^p(G(n,k))$, $1\leq p< \frac{n-k}{d-k}$ then the
integrals involved in (\ref{rg}) is well defined. This range of $p$ is sharp in the sense that if
$p \geq \frac {n-k}{d-k}$ then there exists a radial function $f \in L^p(G(n,k))$ such that $R_df(\zeta)=\infty$
for almost every $\zeta \in G(n,d)$. Our main result in this section deals with behaviour of $R_d$ restricted
to radial functions in $L^{\frac{n-k}{d-k},1}(G(n,k))$. We start by quoting a simple lemma from \cite{KR2}.

\begin{Lemma} Let $n \in \mathbb N$ and $\gamma \geq 1.$ If $x_1 \geq x_2
\geq \ldots \geq x_{n} \geq 0$ are real numbers then the following
inequality holds, \be\left(\sum_{i=1}^{n} (-1)^{i-1}
x_{i}\right)^{\gamma} \leq \sum_{i=1}^{n} (-1)^{i-1}
(x_{i})^{\gamma}.\label{basic} \ee\label{lemma} \end{Lemma} The main
result of this section is the following Theorem.
\bt \label{th2} If
$n>d>k\geq 0,$  then there exists a constant $C>0$ such that for all
radial functions $f$ on $G(n,k)$ the following estimate holds \be
\sup_{s \in (0,\infty)}A_d(f)(s)\leq
C\|f\|_{L^{\frac{n-k}{d-k},1}(G(n,k))}.\label{pestk} \ee
\label{grast}\et
\begin{proof}
In view of general theory of Lorentz spaces it suffices to prove (\ref{pestk})
for indicator functions of radial open sets with finite measure  (see
 \cite{SW3}, Theorem 3.13 and \cite{DNO}). Since these functions can be
approximated by functions of form $\chi_{\cup_{i=1}^lE_i}$ where $E_i=\{\tau \in G(n,k):a_i\leq
|\tau|<b_i\},$ $b_i<a_{i+1},$ $a_1\geq 0$ and $i=1,\ldots,l$ it suffices to prove the result for these functions. An explicit calculation using (\ref{drg}) shows that
\be A_d (\chi_{E_i})(s)=C\begin{cases}\big{[}(b_i^2-s^2)^{\frac
{d-k}{2}}-(a_i^2-s^2)^{\frac {d-k}{2}}\big{]},&\qquad
s<a_i\\
(b_i^2-s^2)^{\frac {d-k}{2}},&\qquad a_i\leq s\leq b_i\\
0,&\qquad b_i<s
\end{cases}\label{exp1}
\ee where $s=|\zeta|,$ $\zeta \in G(n,d).$ If we denote
$E=\cup_{i=1}^lE_i$ then by using linearity of $A_d$ and
(\ref{exp1}) we get that \be
A_d(\chi_{E})(s)=C\begin{cases}\sum_{i=1}^l\big{[}(b_i^2-s^2)^{\frac
{d-k}{2}}-(a_i^2-s^2)^{\frac {d-k}{2}}\big{]},\quad\mbox{if }
s<a_1.\\
\left(\left(\sum_{i=j+1}^l\big{[}(b_i^2-s^2)^{\frac {d-k}{2}}-(a_i^2-s^2)^{\frac {d-k}{2}}\big{]}\right)+(b_j^2-s^2)^{\frac {d-k}{2}}\right),\\
\qquad \mbox{if }a_j\leq s<b_j,  1\leq j \leq l. \\
\sum_{i=j+1}^l\big{[}(b_i^2-s^2)^{\frac {d-k}{2}}-(a_i^2-s^2)^{\frac {d-k}{2}}\big{]},\\
\qquad\mbox{if }
b_j\leq s<a_{j+1},1\leq j \leq {l-1}.\\
0,\qquad b_l < s.\end{cases}\label{exp2} \ee Now, for a fixed $s>0$, let
$E_s$ denotes the set $\{\tau\in E: |\tau|>s\}$ so that
$\chi_{E_s}\leq \chi_E$ and hence the corresponding distribution
functions satisfy the relation $d_{\chi_{E_s}}(r)\leq d_{\chi_E}(r)$
for all $r>0.$ This implies that $\|\chi_{E_s}\|_{L^{\frac{n-k}{d-k},1}(G(n,k))}\leq \|\chi_{E}\|_{L^{\frac{n-k}{d-k},1}(G(n,k))}$. To prove the theorem it thus suffices to show that
for all $s>0$ \be A_d(\chi_E)(s)\leq
C\|\chi_{E_s}\|_{L^{\frac{n-k}{d-k},1}(G(n,k))},\label{basica} \ee
where $C$ is independent of $s$ and $l$.

We now turn towards the calculation of the required Lorentz norm of the function $\chi_{E_s}$.
By using polar coordinates on $\R^{n-k}$ and the expression of the measure on $G(n,k)$ given in the beginning of this section
it follows that \bea \|\chi_{E_s}\|_{L^{\frac{n-k}{d-k},1}(G(n,k))}&=&\mu_k(E_s)^{\frac{d-k}{n-k}}\nonumber\\&=&
C\begin{cases}
\left(\sum_{i=1}^{l} [ {b_i}^{{n-k}} -
{a_i}^{n-k}]\right)^{\frac{d-k}{n-k}},\quad\mbox{if } s < a_1.\\
\left(\left(\sum_{i=j+1}^{l} [ {b_i}^{{n-k}} -
{a_i}^{{n-k}}]\right)+{b_j}^{{n-k}} -{s}^{n-k}\right)^{\frac{d-k}{n-k}},\\\qquad\mbox{if } a_j< s < b_j, 1\leq j \leq l.\\
\left(\sum_{i=j+1}^{l} [ {b_i}^{{n-k}} -
{a_i}^{n-k}]\right)^{\frac{d-k}{n-k}},\\\qquad \mbox{if }b_j<s<a_{j+1}, 1\leq j \leq {l-1}.\\
0,\quad\mbox{if } b_l<s.\end{cases}\eea We will now consider $s\in
(a_j,b_j)$ for a fixed $j$ and prove (\ref{basica}). Other cases can
be dealt with exactly the same way.

For $a>0,$ $\gamma\geq 1$ and  $x\geq y\geq 0$ we apply Lemma
\ref{lemma} to get the inequality
$(x+a)^{\gamma}-x^{\gamma}\geq (y+a)^{\gamma}-y^{\gamma}$. If we now choose $\gamma=({n-k})/2$ (which is greater than or
equal to $1$ by hypothesis) and $x=a_i^2,$ $y=a_i^2-s^2,$
$a=b_i^2-a_i^2$ and use the previous inequality then we get
$${b_i}^{{n-k}} -
{a_i}^{n-k}\geq \big{[}(b_i^2-s^2)^
{\frac{{n-k}}{2}}-(a_i^2-s^2)^{\frac{{n-k}}{2}}\big{]}.$$
Similarly,
$${b_j}^{{n-k}} - {s}^{n-k}\geq (b_j^2-s^2)^{\frac{{n-k}}{2}}. $$
Hence,
\bea && \left(\sum_{i=j+1}^{l} [
{b_i}^{{n-k}} -
{a_i}^{n-k}]+{b_j}^{{n-k}} - {s}^{n-k}\right)^{\frac{d-k}{{n-k}}}\label{estg1}\\
&&\geq \left(\sum_{i=j+1}^l\big{[}(b_i^2-s^2)^
{\frac{{n-k}}{2}}-(a_i^2-s^2)^{\frac{{n-k}}{2}}\big{]}
+(b_j^2-s^2)^{\frac{{n-k}}{2}}\right)^{\frac{d-k}{{n-k}}},\nonumber
\eea for all $s\in (a_j,b_j).$ As $(b_l^2-s^2)^{({n-k})/2}>
(a_l^2-s^2)^{({n-k})/2}>\ldots>(b_j^2-s^2)^{({n-k})/2}$ we can apply
Lemma \ref{lemma} for $\gamma={(n-k)/(d-k)}$ in (\ref{estg1}) to get
\bea &&\left(\sum_{i=j+1}^l\big{[}(b_i^2-s^2)^
{\frac{{n-k}}{2}}-(a_i^2-s^2)^{\frac{{n-k}}{2}}\big{]}
+(b_j^2-s^2)^{\frac{{n-k}}{2}}\right)^{\frac{d-k}{{n-k}}},\nonumber\\
&&\geq\left(\left(\sum_{i=j+1}^l\big{[}(b_i^2-s^2)^{\frac {d-k}{2}}-(a_i^2-s^2)^{\frac {d-k}{2}}\big{]}\right)+(b_j^2-s^2)^{\frac {d-k}{2}}\right)\nonumber.
\eea
The result now follows from \eqref{estg1} and the last inequality.
\end{proof}
\begin{Remark}
{\em We note that if $k=0$ then the estimate (\ref{pestk}) boils down to the estimate (\ref{DNOsynp}) proved in \cite{DNO}}.
\end{Remark}
We will now briefly discuss the $L^p-L^q$ mapping property of $A_d$. We first recall that $R_d$ satisfies the trivial estimate $\|R_df\|_{L^1(G(n,d))} \leq C \|f\|_{L^1(G(n,k))}$
(see \cite{R5}). We can thus use \eqref{pestk} and off-diagonal Marcinkiewicz interpolation theorem (\cite{Graf}, Theorem $1.4.19$) to obtain the following result.
\bt
If $ 0 \leq k < d \leq n-1$, $1\leq p<\frac{n-k}{d-k}$ and $\frac{n-k}{p}=d-k+\frac{n-d}{q},$ then there exists a constant $C>0$ such that
for all radial functions $f \in L^p(G(n,k))$ the following inequality holds,
\be
\|A_df\|_{L^q(G(n,d))}\leq C\|f\|_{L^p(G(n,k))}.\ee
\et

\section[Real Hyperbolic Space]{Real Hyperbolic Space}
\subsection[Notation and Preliminaries]{Notation and Preliminaries}
In this section we will set up notation and explain the basic facts
needed to deal with the Radon transform on real hyperbolic space.
Most of these material is standard and can be found in \cite{BR1}, \cite{BR2}, \cite{Bray},\cite{He2},
\cite{Rj},\cite{R4}, \cite{VK}.

For $n\geq2,$ ${\E}^{n,1}$ denotes the set $\R^{n+1}$ equipped with
the bilinear form \be
[x,y]=-x_1y_1-{\cdots}-x_ny_n+x_{n+1}y_{n+1},\qquad x,y\in
\E^{n,1}.\label{in} \ee We consider the set $F^n=\{x\in
\E^{n,1}:[x,x]=1\}.$ The set $F^n$ is a hyperboloid of two sheets.
The $n\mbox{-}$dimensional real hyperbolic space $\mathbb H^n$ is
defined as the upper sheet of the hyperboloid, that is, \be
\mathbb H^n=\{x \in{\E}^{n,1} : [x,x]=1, x_{n+1}>0\}. \label{defhn}
\ee Let $O(n,1)$ denotes the group of invertible linear transformations of
$\R^{n+1}$ which preserves the bilinear form given in (\ref{in}),
that is,
\begin{eqnarray}
O(n,1)&=&\{g\in GL(n+1,\R): [gx,gy]=[x,y]\ \mbox{for all $x,y\in
\E^{n,1}$}\}\nonumber \\&=&\{g\in GL(n+1,\R):
g^tJg=gJg^t=J\},\label{j}\end{eqnarray} where
$J=\mbox{diag}(-1,\ldots,-1,1)$ (\cite{Rj}, Theorem $3.1.4$). We define
 $SO(n,1)=\{g\in O(n,1): \mbox{det }g=1\}$. Let $O_0(n,1)$ denotes the subgroup of $O(n,1)$ given
by,
$$O_0(n,1)=\{g\in O(n,1): e_{n+1}^tge_{n+1}>0 \}.$$
The subgroup $SO_0(n,1)$ of $O(n,1)$ is defined by, \be
SO_0(n,1)=O_0(n,1)\cap SO(n,1). \ee

It is known that the natural action of $SO_0(n,1)$ on $\mathbb H^n$
is transitive (\cite{Rj}, Theorem $3.1.6$). Let $\{e_1,\ldots
,e_{n+1}\}$ be the standard orthonormal basis of $\R^{n+1}$ and
$x_0=e_{n+1}$ be the origin of $\mathbb H^n.$ Let $K$ be the
isotropy subgroup of $SO_0(n,1)$ at $x_0,$ that is, $K=\{g\in
SO_0(n,1): g(x_0)=g(e_{n+1})=x_0\}$. It turns out that \be
 K= \left\{k=\left [ \begin{matrix}
B & 0 \\
0 & 1 \end{matrix} \right ]:B\in SO(n)\right\}. \label{maxcom}
 \ee
It is well known that $G=SO_0(n,1)$ is a noncompact, connected, rank
one semisimple Lie group and $K$ is a maximal compact subgroup of
$SO_0(n,1)$ (\cite{VK}, p.$2$). It follows that $\mathbb H^n$ is
diffeomorphic to the homogeneous space $SO_0(n,1)/K$ with the
identification $gK\mapsto g.x_0.$ Hence $\mathbb H^n$ is a
Riemannian symmetric space of noncompact type (see \cite{VK}, p. 2).
The $G\mbox{-}$invariant Riemannian metric on $\mathbb H^n$ is given
by \be d(x,y)=\cosh^{-1}([x,y]), \quad x,y \in \mathbb H^n
\label{metric},\ee (\cite{Rj} Theorem $3.2.2$). To proceed further
we need the analogs of $d$-dimensional planes in $\mathbb H^n$. This
requires the notion of bispherical coordinates on $\mathbb H^n.$
Bispherical coordinates are natural generalizations of the notion of
polar coordinates. For $1\leq d\leq n,$ let $\R^{n+1}=R^{n-d}\oplus
R^{d+1}$ where $R^{n-d}=\mbox{span }\{e_1,\ldots, e_{n-d}\}$ and
$R^{d+1}=\mbox{span }\{e_{n-d+1},\ldots, e_{n+1}\}$.
 The following lemma can be found
in \cite{VK}, page $12$.
\begin{Lemma}
If $1\leq d\leq n-1$ then every $x\in\mathbb H^n,$ can be written as
\be x=\zeta\sinh u+\eta\cosh u, \label{coordinate}\ee \ where $0\leq
u<\infty,$ $\zeta\in S^{n-d-1}\subset R^{n-d},$ and $\eta\in \mathbb
H^d\subset R^{d+1}.$ \label{bspco}
\end{Lemma}
We note that if $x\in\mathbb H^n\setminus (\mathbb H^n\cap R^{d+1})$ then the above
representation is unique. In addition, if we consider the case $d=0$
then it follows that every $x\in\mathbb H^n,$  $x\neq x_0$ has a
unique expression \be x=\w\sinh r+e_{n+1}\cosh r,\label{polarh}\ee
where $\w\in S^{n-1}\subset R^n.$ This is the so called polar
coordinate representation of the point $x\in\mathbb H^n.$

The $G\mbox{-}$invariant measure on $\mathbb H^n$ with respect to
the polar coordinates is given by
\be dx=c_n \sinh^{n-1}r\ dr\
d\sigma_{n-1}(\w).\label{measure}\ee

Since our main object of study is the analogue of $d\mbox{-}$plane
transform, we also need a good understanding of the analogues of
$d\mbox{-}$dimensional planes in $\mathbb H^n.$
\begin{Definition}
For $1\leq d\leq n$ we define $\Lambda_{d}$ to be the set of all
linear subspaces $V$ of $\R^{n+1}$ such that:\\
i) $\mbox{dim }V=d+1.$\\
ii) There exists $v\in V$ such that $[v,v]>0.$
\end{Definition}
Since every $V\in\Lambda_d$ contains a vector $v$ with $[v,v]>0$, it
follows that a suitable scalar multiple of $v$ is in $\mathbb H^n$
and hence $\mathbb H^n\cap V\neq \phi.$
\begin{Definition}
$\xi\subset \mathbb H^n$ is called a hyperbolic $d\mbox{-}$plane (or
simply $d\mbox{-}$plane) if there exists $V\in\Lambda_d$ such that
$\xi=\mathbb H^n\cap V.$ The set of all $d\mbox{-}$planes of
$\mathbb H^n$ is denoted by $\Xi_d.$
\end{Definition}
Note that the case $d=1$ describes all the geodesics of $\mathbb
H^n$ (\cite{Rj}, p. $64$). The importance of the $d\mbox{-}$planes
comes from the fact that they are precisely the totally geodesic
submanifolds of $\mathbb H^n,$ that is, the geodesics of these
submanifolds are geodesics of $\mathbb H^n$ (see \cite{Rj}, p.
$72$).

The set $\Xi_d$ can also be thought of as a homogeneous space of the
group $G=SO_0(n,1)$ as follows. We fix a particular $d\mbox{-}$plane
$\xi_0\in\Xi_d$ given by $\xi_0=\mathbb H^n\cap \mbox{
span~}\{e_{n-d+1},\ldots ,e_{n+1}\}$. Using Theorem $3.1.6$ of
\cite{Rj} it follows that the action of $SO_0(n,1)$ on $\Lambda_d$
is transitive and consequently the action of $SO_0(n,1)$ on $\Xi_d$
is also transitive. It is not hard to see that the isotropy subgroup
at $\xi_0$ is isomorphic to $H=SO(n-d)\times SO_0(d,1)$ and hence
$\Xi_d$ is homeomorphic to $G/H$. It follows from the above
discussion that $\xi_0$ can be described as \be \xi_0=\{hx_0: h\in
SO(n-d)\times SO_0(d,1)\}.\label{haction}\ee Since the action of
$SO_0(n,1)$ on $\Xi_d$ is transitive we have that given any
$\xi\in\Xi_d$ there exists $g(\xi)\in SO_0(n,1)$ such that
$g(\xi).\xi_0=\xi$ (note that $g(\xi)$ is unique modulo $H$). So the
element of $\Xi_d$ are nothing but the $G\mbox{-}$translates of
$\xi_0.$ Let $\mu$ denotes the $G\mbox{-}$invariant measure on the
set $\Xi_d$. The explicit expression of the measure $\mu$ is given
by the following lemma (\cite{BR2}, Lemma $2.1$).
\begin{Lemma}
 If $F$ is a nonnegative measurable function on $\Xi_d$ then
 \be
 \int_{\Xi_d}F(\xi)d\mu(\xi)=\int_0^{\infty}\left(\int_K
F(kg_u^{-1}\xi_0)dk\right)(\sinh u)^{n-d-1}(\cosh u)^d du, \ee where
$dk$ denotes the normalized Haar measure on $K$ and \\
\be
 g_u= \left [ \begin{matrix}
\cosh u &      0     & \sinh u \\
   0          &  I_{n-1}   &   0            \\
\sinh u &      0     & \cosh u \end{matrix} \right ]. \label{gtheta}
 \ee \label{xidm}
 \end{Lemma}

We are now in a position to describe the notion of the
$d\mbox{-}$dimensional totally geodesic Radon transform. If $f\in
C_c^{\infty}(\mathbb H^n)$ and $\xi\in\Xi_d$ then the totally
geodesic $d\mbox{-}$dimensional Radon transform (or simply
$d\mbox{-}$plane transform) of $f$ at $\xi$ is defined as \be
R^-_df(\xi)=\int_Hf(g(\xi)h.x_0)\ dh,\label{radon1} \ee where $dh$
is the Haar measure of the group $H$. Since any $\xi$ is of
the form $g.\xi_0$ for some $g\in SO_0(n,1)$ the above definition is
equivalent to \be
R^-_df(g.\xi_0)=\int_Hf(gh.x_0)dh.\label{radon2}\ee

Though we have defined the $d\mbox{-}$plane transform only for
$C_c^{\infty}$ functions but it turns out that the $d\mbox{-}$plane
transform also makes sense for certain class of $L^p$ functions. We
now quote a result from \cite{BR2} which gives a precise description
of this class.
\begin{Theorem}
If $1\leq p<\frac{n-1}{d-1}$ and $f\in L^p(\mathbb H^n)$ then the
integral in (\ref{radon1}) converges for almost every $\xi\in\Xi_d.$
Consequently $R^-_df(\xi)$ is well defined for almost every
$\xi\in\Xi_d.$ Moreover if $p\geq \frac{n-1}{d-1}$ then there exists
a nonnegative radial function $f$ such that $R^-_df(\xi)=\infty$ for
almost every $\xi\in \Xi_d.$
\end{Theorem}

\subsection[The $d\mbox{-}$Plane Transform on Real Hyperbolic
Space]{The $d\mbox{-}$Plane Transform of Radial Functions on
$\mathbb H^n$} In this section we will concentrate mainly on the
class of radial functions. It turns out that, similar to Euclidean spaces, the $d\mbox{-}$plane
transform of a radial function on $\mathbb H^n$ can also be written
down explicitly involving an Abel type integral (\cite{BR2}, Lemma
$3.1$). For the reader's benefit we start with a brief discussion of
the $d\mbox{-}$plane transform of radial functions on $\mathbb H^n$.
A function $f$ on $\mathbb H^n$ is called radial if $f(x)=f(y)$
whenever $d(x_0,x)=d(x_0,y).$ Since the maximal compact subgroup $K$
acts transitively on the unit sphere (see \cite{He3}) it follows
that a function $f$ is radial if and only if $f(k x)=f(x)$ for all
$k \in K$ and $x \in \mathbb H^n$. We have from (\ref{metric}) that
$d(x_0,x)=\cosh^{-1}(x_{n+1})$ and hence a radial function $f$
depends only on the last component of $x,$ that is, on $x_{n+1}.$
Given a radial function $f$ on $\mathbb H^n$ we define a function
$\tilde{f}$ on $[1,\infty)$ by \be \tilde{f}(t)=f(x_1,\ldots,x_n,t),
\quad (x_1,\ldots,x_n,t)\in \mathbb H^n. \label{radial} \ee

The following lemma ( \cite{BR2}, Lemma $3.1$) explicitly describes the $d\mbox{-}$plane
transform of a radial function on $\mathbb H^n$.
\begin{Lemma}\label{rdlem}
If $f\in\C_c^{\infty}(\mathbb H^n)$ is a radial function then for
$\xi\in\Xi_d,$
 \bea R^-_df(\xi)&=&\frac{C_d}{(\cosh s)^{d-1}}
\int_{\cosh s}^{\infty} \tilde{f}(t)(t^2-\cosh^2s)^{\frac{d-2}{2}}\
dt \label{exphyp1}\\&=&\frac{C_d}{\cosh s}
\int_s^{\infty}\tilde{f}(\cosh r)
 \left(1-\frac{\tanh^2s}{\tanh^2r}\right)^{\frac{d-2}{2}} \sinh^{d-1} r\ dr
 \label{exphyp2}\eea
where $s=d(x_0,\xi).$
\end{Lemma}
For a radial function $f$ on $\mathbb H^n$ we define the Abel
transform of $f$ by $A^-_df(s)=R^-_df(\xi)$ where $s=d(x_0,\xi)$. It
is important for us to get an explicit expression of the $L^p$ norm
of $A_d^-f.$ In this regard, we first observe that for $k \in K$,
$$d(x_0,kg_u^{-1}\xi_0)=d(g_uk^{-1}x_0,\xi_0)=d(g_ux_0,\xi_0)=\mbox{inf
}_{y\in\xi_0}d(g_ux_0,y)=d(g_ux_0,x_0)=u.$$ It follows that for
radial functions $R^-_df(kg_u^{-1}\xi_0)=A^-_df(u).$ It now follows
from Lemma \ref{xidm} that if $f$ is radial function on $\mathbb
H^n$ then for $1\leq p<\infty$, \be
\int_{\Xi_d}|R^-_df(\xi)|^pd\mu(\xi)=C_{n,d}
\int_0^{\infty}|A^-_df(u)|^p(\sinh u)^{n-d-1}(\cosh u)^d
du.\label{abellp} \ee Using (\ref{abellp}) and (\ref{measure}) it is
easy to see that (\cite{BR2}, Corollary $2.4$), \be
\|A^{-}_df\|_{L^1(\Xi_d)}\leq C\|f\|_{L^1(\mathbb
H^n)}.\label{beginabelest} \ee

We will now prove a lemma which seems to be a recurring theme as far
as $L^p-L^q$ mapping property of Abel transforms is concerned. The
following lemma is implicit in the proof of (\ref{DNOsynp}) as well as in the proof of Theorem \ref{grast}. A
special case of the lemma also played a crucial role in the
end-point estimate of the horospherical Radon transform of radial
functions on rank one symmetric spaces of noncompact type
(\cite{Io}, Lemma $3$). Though we are interested in the totally geodesic $d\mbox{-}$dimensional Radon transform (instead of the horospherical Radon transform) of
radial functions on $\mathbb H^n$ and $S^n,$ it turns out that this
lemma is still an essential ingredient for the results to follow.

\begin{Lemma}
 If $\delta \neq 0$ then there exists a positive constant $C_{\delta}$ such that
 for all indicator functions of measurable subsets of $\mathbb R$ the following inequality
 holds,
\be \int_{\mathbb R}f(t)e^{\delta t}dt \leq C_{\delta}
\left(\int_{\mathbb R}f(t)e^{p\delta
t}dt\right)^{1/p}.\label{ionlemma}\ee \label{ionescu}
\end{Lemma}
\begin{proof}
Using the change of variable $e^{\delta t}=s$ it suffices to prove
the inequality \be \int_0^{\infty}\phi (s)ds\leq
C_{\delta}\left(\int_{0}^{\infty}\phi
(s)s^{p-1}ds\right)^{1/p},\label{ionreduced} \ee for all indicator
functions $\phi$ of measurable subsets of $(0,\infty).$ We first
consider the case when $\phi$ is of the form
$\chi_{\cup_{i=1}^k[a_i,b_i]}$ with $a_1<b_1\leq
a_2<b_2\leq\cdots\leq a_k<b_k.$ For functions of this form
(\ref{ionreduced}) follows immediately from (\ref{basic}) by
considering $x_1=b_k, x_2=a_k,\ldots, x_{2k-1}=b_1, x_{2k}=a_1.$
Since every nonempty open subset of $\R$ is a countable disjoint
union of open intervals, the result holds for indicator function of
open sets by monotone convergence theorem. Now suppose that $E$ is
a measurable subset of $(0,\infty).$ By monotone convergence theorem
it suffices to prove the result for $E\cap (0,m)$ for each
$m\in\mathbb N.$ So without loss of generality we can assume that
$E\subset (0,\lambda)$ for some $\lambda>0.$ By regularity of the
Lebesgue measure we can get a decreasing family of open sets
$\{U_n\}$ such that $E \subset U_n\subset (0,\lambda)$ for each $n$
and $|U_n\setminus E|<\frac{1}{n}$ (here $|E|$ denotes Lebesgue
measure of the set $E$). So \be |E|=|U_n|-|U_n\setminus
E|\leq \left(\int_0^{\infty}\chi_{U_n}(s)s^{p-1}\
ds\right)^{1/p}-|U_n\setminus E|.\label{lemmastep}\ee It now suffices
to show that
$$ \lim_{n\rightarrow\infty}\int_0^{\infty}\chi_{U_n}(s)s^{p-1}\
ds=\int_0^{\infty}\chi_{E}(s)s^{p-1}\ ds.$$ If $d\nu(s)=s^{p-1}ds$
then
\begin{eqnarray*}
&&\lim_{n\rightarrow\infty}\int_0^{\infty}\chi_{U_n}(s)s^{p-1}\
ds-\int_0^{\infty}\chi_{E}(s)s^{p-1}\ ds\\&=&
\lim_{n\rightarrow\infty}\int_0^{\infty}\chi_{U_n\setminus
E}(s)s^{p-1}ds\\&=&\nu\left(\cap_{n=1}^{\infty}(U_n\setminus
E)\right).\end{eqnarray*} Since Lebesgue measure of the set
$\cap_{n=1}^{\infty}(U_n\setminus E)$ is zero it follows that
$$\lim_{n\rightarrow\infty}\int_0^{\infty}\chi_{U_n\setminus E}(s)s^{p-1}ds=0.$$
\end{proof}
\begin{Remark}{\em Using the substitution t=logs it follows from (\ref{ionlemma}) that
 \be \int_0^{\infty}\phi (s)\  s^{\delta-1}\ ds\leq
C_{\delta}\left(\int_{0}^{\infty}\phi (s)\ s^{p\delta-1}\
ds\right)^{1/p}. \label{ioneskure}\ee  }\label{ioneskure}
\end{Remark}
We are now in a position to state and prove the main result of
 this section.
 \bt
 \begin{enumerate}
 \item[a)] If $n\geq 3$ and $2\leq d \leq n-1$ then there exists a constant $C>0$ such that for all measurable, radial
 functions $f$ on $\mathbb H^n$ the following inequality holds,
 \be
 \|A^-_df\|_{L^{\infty}(\Xi_d)} \leq C \|f\|_{L^{\frac{n-1}{d-1},1}(\mathbb H^n)}.\label{endabelest}
 \ee
 \item[b)] If $n\geq 2$ and
$1\leq d\leq n-1$ then there exists a constant $C>0$ such
that for all measurable radial functions on $\mathbb H^n$ the
following inequality holds, \be \|A^-_df\|_{L^q(\Xi_d)}\leq C
\|f\|_{L^p(\mathbb H^n)},\label{endresult}\ee where
\be\frac{n-1}{p}=d-1+\frac{n-d}{q},\quad 1\leq
p<\frac{n-1}{d-1}.\label{endp}\ee \end{enumerate} \label{endabel}
 \et
 \begin{proof} We will first prove part $a)$ of the theorem. As in Theorem \ref{grast} it is enough to prove the result for
 indicator functions of all measurable, radial subsets of $\mathbb H^n$ (see
 \cite{SW3}, Theorem 3.13 ).
 We will show that there exists
 a constant $C>0$ such that
 \be |\cosh (s) A^-_df(s)|\leq C \left(\int_{\mathbb H^n} f(x)\ dx\right)^{\frac{d-1}{n-1}}
 = C \left(\int_0^{\infty} \tilde{f}(\cosh t) \sinh ^{n-1} t\ dt
 \right)^{\frac{d-1}{n-1}},\label{firstend}\ee
 where $s\in (0,\infty)$ and $f$ is indicator function of a radial,
 measurable subset of $\mathbb H^n.$ Using (\ref{exphyp2}) we write
 \be
\cosh (s) A^-_df(s)=\sigma_{d-1} \int_s^{\infty} \tilde{f}(\cosh t)
\left(1-\frac{\tanh^2s}{\tanh^2t}\right)^{\frac{d-2}{2}}
 \sinh^{d-1} t\ dt=
 I_1+I_2,\nonumber
 \ee
 where
 \bea
 I_1&=&\sigma_{d-1}\int_s^{s + \alpha} \tilde{f}(\cosh t) \left(1-\frac{\tanh^2s}{\tanh^2t}\right)^{\frac{d-2}{2}}
 \sinh^{d-1} t\ dt,\\
 I_2&=&\sigma_{d-1}\int_{s + \alpha}^{\infty} \tilde{f}(\cosh t) \left(1-\frac{\tanh^2s}{\tanh^2t}\right)^{\frac{d-2}{2}}
 \sinh^{d-1} t\ dt,
 \eea
and $\alpha=1/2$.
We will estimate $I_1$ and $I_2$ separately.
 For $t\in (s + \alpha,\infty)$ we have $\tanh s < \tanh t$ and $\sinh t\asymp e^t$. So
\be
 I_2\leq
  C \int_{s + \alpha}^{\infty} \tilde{f}(\cosh t) e^{(d-1)t}\ dt.
\ee We now appeal to Lemma \ref{ionescu} with $\delta=d-1$ and
$p=\frac{n-1}{d-1}$ to get \bea
  \int_{s + \alpha}^{\infty} \tilde{f}(\cosh t) e^{(d-1)t}\ dt&\leq &C\left(\int_{s + \alpha}^{\infty} \tilde{f}(\cosh t)
  e^{(n-1)t} dt \right)^{\frac{d-1}{n-1}}\nonumber\\
  &\asymp & C \left(\int_{s + \alpha}^{\infty} \tilde{f}(\cosh t) \sinh^{n-1}t\  dt
  \right)^{\frac{d-1}{n-1}}\nonumber\\
  &\leq &C \|f\|_{L^{\frac{n-1}{d-1},1}(\mathbb H^n)}.
  \eea
 By similar argument as above we can prove the required estimate for $I_1$ if $s \geq \alpha.$ We will now estimate $I_1$ when $0<s<\alpha.$
 We first observe that by mean value theorem that there exists a real number $u\in (s,t)$
such that
  \bea
  &&\int_s^{s + \alpha} \tilde{f}(\cosh t) (\tanh^2t-\tanh^2s)^{\frac{d-2}{2}} \coth^{d-2} t\ \sinh^{d-1} t\ dt\nonumber\\
  &&=\int_s^{s + \alpha} \tilde{f}(\cosh t) (t-s)^{\frac{d-2}{2}}
  \left(\frac{2\tanh u }{\cosh^{2} u}\right)^{\frac{d-2}{2}} \frac{ \sinh^{d-1} t}{\tanh^{d-2} t}\ dt \label{mvt}
  \eea

  If $0<r<1$ then we have $\sinh r \asymp r$ and $\cosh r \asymp 1$. As $s<u<t<s+\alpha<1$ and $d\geq 2$ it follows from
 the Remark \ref{ioneskure} and (\ref{mvt}) that
\bea
I_1&\leq&C \int_s^{s + \alpha} \tilde{f}(\cosh t) (t-s)^{\frac{d-2}{2}} (u)^{\frac{d-2}{2}}  t\ dt\nonumber\\
   &\leq&C \int_s^{s + \alpha} \tilde{f}(\cosh t) (t-s)^{\frac{d-2}{2}} t^{\frac{d}{2}}\ dt\nonumber\\
   &\leq&C \int_s^{s + \alpha} \left(\tilde{f}(\cosh t)t^{d-1}\right) \left((t-s)^{\frac{d-2}{2}} t^{\frac{2-d}{2}}\right)\ dt\nonumber\\
   &\leq&C \int_s^{s + \alpha}\tilde{f}(\cosh t) t^{d-1} \ dt\nonumber\\
   &\leq&C \int_s^{s + \alpha}\tilde{f}(\cosh t) t^{n(d-1)/(n-1)-1} \ dt \nonumber\\
   &\leq&C \left(\int_s^{s + \alpha}\tilde{f}(\cosh t) t^{n-1}\ dt\right)^{\frac{d-1}{n-1}}\nonumber\\
   &\leq&C \left(\int_s^{s + \alpha} \tilde{f}(\cosh t) \sinh^{n-1}t
   dt\right)^{\frac{d-1}{n-1}}\nonumber\\
   &\leq&C \|f\|_{L^{\frac{n-1}{d-1},1}(\mathbb H^n)}.
  \eea
  This proves (\ref{firstend}). As $\cosh s\geq 1$ this completes the proof of $a)$.

The proof of $b),$ for $2\leq d \leq n-1,$ follows by interpolating
(\cite{Graf}, Theorem $1.4.19$) between the estimates
(\ref{endabelest}) and (\ref{beginabelest}).

 For $d=1,$ the relation (\ref{endp}) shows that we need to prove the inequality
\be
 \|A^-_1f\|_{L^{p}(\Xi_d)} \leq C \|f\|_{L^{p}(\mathbb H^n)},\quad 1 \leq
 p<\infty.
 \ee
We first write the formula for $A^-_1f$ given
 by (\ref{exphyp1}) as follows
$$A^-_1f(s)=C\int_{\cosh s}^{\infty} \tilde{f}(t)(t^2-\cosh^2s)^{-\frac{1}{2}}dt.$$
From (\ref{abellp}) we have \bea \|A^-_1f\|_{L^p(\Xi_1)}&=&C
\left(\int_0^{\infty} |A^-_1f(s)|^p \sinh^{n-2} s \cosh s \
ds\right)^{1/p}\nonumber\eea
 \bea &=&C \left(\int_0^{\infty}
\left(\int_{\cosh s}^{\infty}
\tilde{f}(t)(t^2-\cosh^2s)^{-\frac{1}{2}}\ dt\right)^p \sinh^{n-2} s
\cosh s \ ds\right)^{1/p}, \nonumber\eea (see Lemma \ref{rdlem}). We
note that $\tilde{f}$ is a positive function. Using the substitution
$\cosh s =r$, we get \bea \|A^-_1f\|_{L^p(\Xi_1)}&=&C
\left(\int_1^{\infty} \left(\int_{r}^{\infty}
\tilde{f}(t)(t^2-r^2)^{-\frac{1}{2}}dt\right)^p
(r^2-1)^{\frac{n-3}{2}} r dr\right)^{1/p}\nonumber \eea Again, using
the change of variables $t^2=v+1, r^2=u+1$, we get \bea
\|A^-_1f\|_{L^p(\Xi_1)}&=&C \left(\int_0^{\infty}
\left(\int_{u}^{\infty} \tilde{f}\left(\sqrt {v+1}\right)(v-u)^{-\frac{1}{2}} (v+1)^{-\frac 12}\ dv\right)^p \ u^{\frac{n-3}{2}} du\right)^{1/p}\nonumber\\
&=&C\left(\int_0^{\infty} \left(\int_{0}^{\infty}
\tilde{f}\left(\sqrt {v+1}\right)(v+1)^{-\frac {1}{2p}}v^{\frac
{n}{2p}}
\left(\frac{v}{1+v}\right)^{\frac {1}{2p^{\prime}}}\right.\right.\nonumber\\
&&\times
\left.\left.\left(1-\frac{u}{v}\right)^{-\frac{1}{2}}({u}/{v})^{\frac{n-1}{2p}}\chi_{[0,1]}({u}/{v})\frac{dv}{v}\right)^p
 \frac{du}{u}\right)^{1/p}\nonumber\\
&\leq&C\left(\int_0^{\infty} \left(\int_{0}^{\infty} \tilde{f}\left(\sqrt {v+1}\right)(v+1)^{-\frac {1}{2p}}v^{\frac{n}{2p}}\right.\right.\nonumber\\
&&\times
\left.\left.\left(1-\frac{u}{v}\right)^{-\frac{1}{2}}({u}/{v})^{\frac{n-1}{2p}}\chi_{[0,1]}({u}/{v})\
\frac{dv}{v}\right)^p \frac{du}{u}\right)^{1/p},\quad
\left(\mbox{as}\ \frac{v}{1+v} <1 \right).\nonumber \eea By Young's
inequality for the group $(0,\infty),$ we get, \bea
\|A^-_1f\|_{L^p(\Xi_1)}&\leq&C\left(\int_0^{\infty}
\left(\tilde{f}\left(\sqrt {v+1}\right)\right)^p
\ (v+1)^{-\frac 12}v^{\frac n2}\frac{dv}{v}\right)^{1/p}\nonumber\\
&&\quad \times \left(\int_0^{\infty}
(1-u)^{-\frac{1}{2}}u^{\frac{n-1}{2p}}\chi_{[0,1]}(u)\frac{du}{u}\right)
\nonumber\\
& \leq&C\left(\int_1^{\infty} (\tilde{f}(s))^p (s^2-1)^{\frac{n-2}{2}} ds\right)^{1/p}\nonumber\\
&&\qquad (\mbox{by using }v+1=s^2 )\nonumber\\
& \leq&C\left(\int_0^{\infty} (\tilde{f}(\cosh r))^p \sinh^{n-1}r dr \right)^{1/p}\nonumber\\
&& \quad (\mbox{by using}\ s=\cosh r) \nonumber\\
&=& C\|f\|_{L^p(\mathbb H^n)} \eea This completes the proof.
 \end{proof}
\begin{Remark} {\em  Though Theorem \ref{endabel} is analogous
to the corresponding result on Euclidean space (\cite{DNO},  Theorem
 $1$) but it does not seem to reveal the full story. It turns out that the
$L^p-L^q$ mapping property of the Abel transform on $\mathbb H^n$ is
very different from that of Euclidean space. In this regard, we
observe that if $d\geq 2$ then (\ref{firstend}) shows that if $f\in
L^{\frac{n-1}{d-1},1}(\mathbb H^n)$ then $$|A^-_df(s)|\leq C_f(\cosh
s)^{-1},\quad s>0,$$ where $C_f$ is a constant multiple of
$\|f\|_{L^{\frac{n-1}{d-1},1}(\mathbb H^n)}.$ Since the function
$g(s)=1/\cosh s$ decays like $e^{-s}$ at infinity and the
$G\mbox{-}$invariant measure on $\Xi_d$ grows like $e^{(n-1)s}$ at
infinity (see Lemma \ref{xidm}) it follows that $g \in
L^{n-1,\infty}(\Xi_d,\mu).$ As a consequence, we have \be
\|A^-_df\|_{L^{n-1,\infty}(\Xi_d)}\leq
C\|f\|_{L^{\frac{n-1}{d-1},1}(\mathbb H^n)}. \label{secondend} \ee }
\end{Remark}
\begin{Corollary}
If $ 2\leq d\leq n-1$, $1\leq p<\frac{n-1}{d-1}$ and $\frac{n-\kappa}{p}=(d-\kappa)+\frac{n-d}{q_{\kappa}},$ then for all
$\kappa\in[1,2]$ the following inequality holds, \be
\|A^-_df\|_{{q_{\kappa}}(\Xi_d)}\leq C_{\kappa}\|f\|_{L^p(\mathbb
H^n)},\ee
for all radial measurable function on $\mathbb H^n.$ \label{coroh}
\end{Corollary}
\begin{proof}
From (\ref{endp}) we have \be \|A^-_df\|_{L^q(\Xi_d)}\leq C
\|f\|_{L^p(\mathbb H^n)},\label{firstq_1}\ee
\be\frac{n-1}{p}=d-1+\frac{n-d}{q},\quad 1\leq
p<\frac{n-1}{d-1}.\label{q_1}\ee By interpolating between the
estimates (\ref{beginabelest}) and (\ref{secondend}) we get \be
\|A^-_df\|_{L^q(\Xi_d)}\leq C \|f\|_{L^p(\mathbb
H^n)},\label{firstq_2}\ee \be\frac{n-2}{p}=d-2+\frac{n-d}{q},\quad
1\leq p<\frac{n-1}{d-1}.\label{q_2}\ee If we fix $p\in
(1,\frac{n-1}{d-1})$ then there exist $q_1$ and $q_2$ such that
$$\|A^-_df\|_{L^{q_j}(\Xi_d)}\leq C \|f\|_{L^p(\mathbb H^n)},\qquad j=1,2,$$
where $q_1$ and $q_2$ are given by (\ref{q_1}) and (\ref{q_2})
respectively. It is easy to see from above that $q_2<q_1.$ Hence
elements of $[q_2,q_1]$ can be written as
$$\frac{1}{q_{\kappa}}=\frac{2-\kappa}{q_1}+\frac{\kappa-1}{q_2},\quad \kappa\in [1,2].$$
It is now easy to see that $q_{\kappa}$ satisfies the relation,
$$\frac{n-\kappa}{p}=(d-\kappa)+\frac{n-d}{q_{\kappa}}.$$ Since
$q_2<q_{\kappa}<q_1$ we have (see \cite{Graf}, Proposition $1.1.14$)
$$\|A^-_df\|_{L^{q_{\kappa}}(\Xi_d)}\leq C_{p,q_1,q_2}\|A^-_df\|_{L^{q_2}(\Xi_d)}^{\frac{\frac{1}{q_{\kappa}}-
\frac{1}{q_1}}{\frac{1}{q_2}-\frac{1}{q_1}}}
\|A^-_df\|_{L^{q_1}(\Xi_d)}^{\frac{\frac{1}{q_2}-\frac{1}{q_{\kappa}}}{\frac{1}{q_2}-\frac{1}{q_1}}}.$$
The result now follows by applying (\ref{firstq_1}) and
(\ref{firstq_2}).
\end{proof}

\section[The Sphere]{The Sphere}
\subsection[Notation and Preliminaries]{Notation and Preliminaries}
 In this section we will discuss
about the notion of the $d\mbox{-}$dimensional totally geodesic
Radon transform on the unit sphere. Since the situation here is
analogous to that of $\mathbb H^n$ our exposition will be brief. We
start with a few notation. Let \bea S^n&=&\left\{x=(x_1,\ldots
,x_{n+1})\in\R^{n+1}:
\ds \sum_{i=1}^{n+1}x_i^2=1\right\},\nonumber\\
R^{n-d}&=&\mbox{span }\{e_1,\ldots ,e_{n-d}\},\nonumber\\
R^{d+1}&=&\mbox{span }\{e_{n-d+1},\ldots ,e_{n+1}\},\nonumber\\
\xi_0&=&R^{d+1}\cap S^n=S^d,~~~x_0=e_{n+1},\nonumber \eea where
$\{e_1,\ldots ,e_{n+1}\}$ is the standard orthonormal basis of
$\R^n$ and $1\leq d\leq n-1.$ We note that the situation here is
little different from that of $\mathbb H^n,$ in the sense that $x\in
\xi_0$ if and only if $-x\in \xi_0.$ The compact Lie group
$G=SO(n+1)$ acts transitively on $S^n$ and the isotropy subgroup at
$x_0$ is given by
$$K=\left\{\left(\begin{array}{cc}k &0\\0 &1
\end{array}\right): k\in SO(n)\right\} \approx SO(n).$$ Hence $S^n$ can be viewed as the homogeneous space $G/K.$

It is known that all $d\mbox{-}$dimensional totally geodesic
submanifolds of $S^n$ are intersections of $S^n$ with
$(d+1)\mbox{-}$dimensional subspaces of $\R^{n+1}$ (\cite{Rj}, p.
$40$). Hence the set of $d\mbox{-}$dimensional totally geodesic
submanifolds of $S^n$ can be parametrized by $G_{n+1,d+1}.$ We note
that $\xi_0\in G_{n+1,d+1}$ and for every $\xi\in G_{n+1,d+1},$
$x\in\xi$ if and only if $-x\in\xi.$ The group $G$ also acts
transitively on $G_{n+1,d+1}$ with the isotropy subgroup at $\xi_0$
given by
$$H=\left\{h=\left(\begin{array}{cc}S &0\\0 &T
\end{array}\right): T\in SO(d+1), S\in SO(n-d)\right\}\approx SO(n-d)\times SO(d+1).$$
Thus $G_{n+1,d+1}$ is also a homogeneous space of the group
$SO(n+1),$ namely, $G_{n+1,d+1}\approx G/H$ (\cite{R3}, p. $78$). In
view of the above discussion it is now easy to see that if $\xi\in
G_{n+1,d+1}$ then there exists a $g(\xi)\in SO(n)$ (which is unique
modulo $H$) such that

$$ \xi_0=\{hx_0:h\in H\},\quad\xi=g(\xi)\xi_0=\{g(\xi)hx_0:h\in H\}.$$ We are now in a position to define the
notion of $d\mbox{-}$dimensional totally geodesic Radon transform
($d\mbox{-}$plane transform) on the sphere.
\begin{Definition}
Given a continuous function $f$ defined on $S^n$ we define the
$d\mbox{-}$dimensional totally geodesic Radon transform of $f$ as
\be R^+_df(\xi)=\int_Hf(g(\xi)hx_0)\ d h,\qquad\xi\in
G_{n+1,d+1}\label{radonsp} \ee where $dh$ stands for the normalized
Haar measure on the compact group $H.$
\end{Definition}

Using the identification of $\xi_0$ with $S^d$ one can see that
(\ref{radonsp}) can also be written as\be
R^+_df(\xi)=\int_{S^d}f(g(\xi)y)\ d\sigma_d(y).\label{radonsp1}\ee
\begin{Remark}
{\em We note the following important difference between
$d\mbox{-}$plane transform on $\mathbb H^n$ and $S^n.$ As $\xi$ is
invariant under reflection about the origin it follows from
(\ref{radonsp}) that $R^+_df(\xi)=0$ for all $\xi\in G_{n+1,d+1}$ if
$f$ is an odd function.\label{odd}}
\end{Remark}
We will now specialize to the class of radial functions on $S^n$. An
explicit formula for $d\mbox{-}$plane transform of radial functions
appear in \cite{R3} (see also \cite{He2}). For the sake of
completeness we explain it in some detail.
\begin{Definition}
A function $f$ defined on $S^n$ is called radial (or zonal) if for
all $x\in S^n$ and for all $k\in K,$ $f$ satisfies the condition
$f(kx)=f(x).$
\end{Definition}
To understand the radial functions we use the notion of polar
coordinate on the sphere. Every element $x\in S^n,$ with $x\neq x_0$
can be uniquely written as \be x=\rho\sin\theta
+x_0\cos\theta,\qquad \rho\in S^{n-1},
0<\theta\leq\pi,\label{polarsp}\ee or equivalently
$x=k.(\sin\theta e_1+x_0\cos\theta),\qquad k\in K.$ For radial
functions $f$ it now follows that $f(x)=f(\rho\sin\theta
+x_0\cos\theta)=f(\sin\theta e_1+x_0\cos\theta).$ Hence a radial
functions $f$ on $S^n$ can be thought of as a function $\tilde{f}$ on the interval
$[0,\pi]$ given by the relation \be \tilde{f}(\cos\theta)=f(x)=f(\rho\sin\theta
+x_0\cos\theta).\label{radialsp}\ee The Riemannian metric on $S^n$
is given by $d_2(x,y)=\cos^{-1}(\langle x,y\rangle),$ (\cite{Rj},
p. $36$). Since $d_2(\rho\sin\theta +x_0\cos\theta,x_0)=\theta,$ it
follows that a radial function on $S^n$ is actually a function of
the distance of a point from $x_0.$
To proceed further we need the notion of bispherical coordinate on
$S^n$ (\cite{VK}, p. $23$). If $1\leq d\leq n-1$ then every $x\in
S^n$ can be written as \be x=\eta\cos\theta
+\zeta\sin\theta,\label{bipolarsp}\ee where $\eta\in S^d=R^{d+1}\cap
S^n,$ $\zeta\in S^{n-d-1}=R^{n-d}\cap S^n,$ and $\theta\in
[0,\frac{\pi}{2}]$. In these coordinates the $G$ invariant measure
on $S^n$ is given by \be dx=\sin^{n-d-1}\theta \cos^{d}\theta \
d\theta \ d\eta\  d\zeta,\label{measurehypersp}\ee where $d\eta$ and
$d\zeta$ denote the normalized rotation invariant measures on $S^d$
and $S^{n-d-1}$ respectively (\cite{VK},p. $12,$ $22$). Now, suppose
that $f$ is a radial $C^{\infty}$ function on $S^n.$ Then for
$\xi\in G_{n+1,d+1}$ we have from (\ref{radonsp1}) \bea
R_d^+f(\xi)&=&\int_{S^d}f(g(\xi)y)\ d\sigma_d(y)\nonumber\\
&=&\int_{S^d}\tilde{f}(\langle g(\xi)y,x_0\rangle)\ d\sigma_d(y)\nonumber\\
&=&\int_{S^d}\tilde{f}(\langle y,
g(\xi)^{-1}x_0\rangle)\ d\sigma_d(y)\nonumber\\
&=&\int_{S^d}\tilde{f}(\langle y,\eta\cos\theta\rangle )\
d\sigma_d(y).\nonumber \eea In the last step we have used the
bispherical representation
$g(\xi)^{-1}x_0=\eta\cos\theta+\zeta\sin\theta$ (see
(\ref{bipolarsp})).
To make the above formula more explicit we will need the catalan formula which is described below. let $\psi$ be a function defined on $\R$ and let $x\in \R^d$, $d\geq 2$. Then
\be
\int_{S^d}\psi(\langle x,\omega\rangle) d\sigma\omega=C_d\int_{-1}^1\psi (s\|x\|)(1-s^2)^{\frac{d-2}{2}}ds.\label{catalan}
\ee
For proof of this formula we refer the reader to \cite{Graf}, D.3.
Thus \bea
R_d^+f(\xi)&=&\int_{S^d}\tilde{f}(\cos\theta\langle y,\eta\rangle
)\ d\sigma_d(y)\nonumber\\
&=&C\int_{-1}^1\tilde{f}(\cos\theta s)(1-s^2)^{\frac{d-2}{2}}\ ds\quad\mbox{(by using \eqref{catalan})}\nonumber\\
&=&C\int_{-\cos\theta}^{\cos\theta}\frac{\tilde{f}(u)}{\cos\theta}\left(1-\frac{u^2}{\cos^2\theta}\right)^{\frac{d-2}{2}}\
du
\quad\mbox{(using $s\cos\theta=u$)}\nonumber\\
&=&\frac{C}{(\cos\theta
)^{d-1}}\int_{-\cos\theta}^{\cos\theta}\tilde{f}(u)(\cos^2\theta-u^2)^{\frac{d-2}{2}}\
du. \label{abelsp} \eea It follows from (\ref{abelsp}) that for
radial functions $R_d^+f$ is a function of $\theta$ only.

In the following the $d\mbox{-}$plane transform of radial functions
$f$ will be denoted by $A_d^+f$ and will be called the Abel
transform of $f.$
\begin{Lemma}
If $f\in C^{\infty}(S^n)$ is an even, radial function then for
$\theta\in [0,\frac{\pi}{2}],$ \bea A^+_df(\theta)&=&
\frac{C}{(\cos\theta
)^{d-1}}\int_0^{\cos\theta}\tilde{f}(u)(\cos^2\theta-u^2)^{\frac{d-2}{2}}du,\label{abelsph1}\\
&=&\frac{C}{\cos\theta}\int_{\theta}^{{\pi}/{2}}\tilde{f}(\cos
r)\left(1-\frac{\tan^2\theta}{\tan^2r}\right)^{\frac{d-2}{2}}\sin^{d-1}
rdr.\label{abelsph2} \eea\label{abelspL}
\end{Lemma}
\begin{proof}
The first identity follows from (\ref{abelsp}) by using the fact
that $f$ is even. The second identity follows from the first one by
using the change of variable $u=\cos r.$
\end{proof}
We end this section by quoting the following result from \cite{R3}
regarding the mapping property of $R_d^+.$ \bt If $1\leq p\leq
\infty$ and $1\leq d\leq n-1$ then for all continuous functions $f$
on $S^n,$ \be \|R_d^+f\|_{L^p(G_{n+1,d+1})}\leq
C\|f\|_{L^p(S^n)}.\label{ppsph}\ee \label{rubinsp} \et

\subsection[The $d\mbox{-}$Plane Transform on the Sphere]{The
$d\mbox{-}$Plane Transform of Radial Functions on the Sphere}

As in the previous section we are interested in the inequalities of
the form $$\|A^+_df\|_{L^q(G_{n+1,d+1})}\leq C\|f\|_{L^p(S^n)},$$
which should be valid for all $f\in C^{\infty}(S^n).$ Any $f\in
C^{\infty}(S^n)$ can be written as $f=f_1+f_2$ where
$f_1(x)=(f(x)+f(-x))/2$ is an even function and $f_2$ is an odd
function. Consequently $A^+_df=A^+_df_1$ (see Remark \ref{odd}). To
prove an inequality of the above form it thus suffices to prove an
inequality of the form $\|A^+_df_1\|_{L^q(G_{n+1,d+1})}\leq
C\|f_1\|_{L^p(S^n)},$ as
$$\|A^+_df\|_{L^q(G_{n+1,d+1})}=\|A^+_df_1\|_{L^q(G_{n+1,d+1})}\leq C\|f_1\|_{L^p(S^n)}\leq C\|f\|_{L^p(S^n)}.$$
So, from now onwards, we will deal only with nonnegative, even
functions on $S^n.$ We start with an example to show that, situation
here is different from that of $\R^n$ and $\mathbb H^n.$

\begin{Example} {\em We will show that $A_d^+$ is not bounded from $L^{p,1}(S^n)$ to $L^{\infty}(G_{n+1,d+1})$ if
$p<\infty$. Let
$$f_i(s)=\chi_{[0,a_i]}(s),\qquad 0<a_i<1, $$
and $\{a_i\}$ be a decreasing sequence converging to $0.$ If $a_i > \cos s$ then from
(\ref{abelsph1}) we get \be A_d^+f_i(s)=C(\cos s)^{1-d} \int_0^{\cos
s} (\cos^2s-t^2)^{\frac{d-2}{2}}\ dt.\label{afi}\ee It follows from (\ref{afi}) that the sequence
$\{\|A_d^+f_i\|_{L^{\infty}(G_{n+1,d+1})}\}$ is bounded away from
zero. On the other hand, \bea
\|\chi_{[0,a_i]}\|_{L^{p,1}(S^n)}&=&C \left(\int_0^{\pi/2} \chi_{[0,a_i]}(\cos s)\ \sin^{n-1} s\ ds\right)^{1/p}\nonumber\\
&=&C \left(\int_0^{1} \chi_{[0,a_i]}(t)\ (1-t^2)^{\frac{n-2}{2}}\ dt \right)^{1/p}\nonumber\\
&\leq &C\left(\int_0^{1} \chi_{[0,a_i]}(t)\ (1-t)^{\frac{n-2}{2}}\ dt \right)^{1/p}\nonumber\\
&=&C(1-(1-a_i)^{n/2})^{1/p}.\label{chai} \eea From (\ref{chai}) it
is clear that the sequence
$\{\|\chi_{[0,a_i]}\|_{L^{p,1}(S^n)}\}_{i=1}^{\infty}$ converges to
$0.$ This implies that Abel transform cannot be bounded from
$L^{p,1}(S^n)$ to $L^{\infty}(G_{n+1,d+1})$ if $p<\infty.$}
\label{examsphere}\end{Example}

The following theorem can be considered as an analogue of
(\ref{firstend}). \bt If $1\leq d \leq n-1$ then for all
non-negative $K$ invariant function $f$ on $S^n,$ there exists
 a positive constant $C$ such that
 \be
 \|\cos{(\cdot)}A_d^+f{(\cdot)}\|_{L^{\infty}(G_{n+1,d+1})} \leq C \|f\|_{L^{\frac{n}{d},1}(S^n)}.\label{endabelestsph}
 \ee
 \label{endabelsph}
 \et
 \begin{proof} As in Theorem \ref{endabel}, it is sufficient to prove the result for functions of the form $f(t)=\chi_{E}(t)$,
  where $E=\bigcup_{i=1}^{m} [a_i,b_i]$ ,$m \in \mathbb N$ and $0 \leq a_1 \leq  b_1 \leq a_2 \leq \cdots \leq a_l\leq b_l \leq 1.$
  We will first prove the result for the case $2 \leq d \leq n-1$.
 Using (\ref{abelsph2}) we write
 $$
 \cos s A_d^+f(s)=C \int_{s}^{\pi/2} f(\cos r) \left(1-\frac{\tan^2s}{\tan^2r}\right)^{\frac{d-2}{2}}\ \sin^{d-1} r\ dr .
 $$
 Since $s \leq r $, we have $0 \leq 1-\frac{\tan^2s}{\tan^2r} \leq 1$. Now from the above expression we have
 \bea
 \cos s A_d^+f(s)&\leq& C \int_{s}^{\pi/2} f(\cos r)\  \sin^{d-1} r\ dr\nonumber\\
 &\leq& C \int_{0}^{\pi/2} f(\cos r)\  r^{d-1}\  dr\nonumber
 \eea
  By Lemma \ref{ionescu} (see also Remark \ref{ioneskure}) we get
 \bea
 \cos s A_d^+f(s)&\leq& C \left(\int_{0}^{\pi/2} f(\cos r)\  r^{n-1}\  dr\right)^{d/n}\nonumber\\
 &\leq& C \left(\int_{0}^{\pi/2} f(\cos r)\  \sin^{n-1}r \ dr\right)^{d/n}\nonumber\\
 &\leq& C \|f\|_{L^{\frac nd,1}(S^n)}.\nonumber
 \eea
  We will now prove the case $d=1.$ From (\ref{abelsph1}) we have
 $$
 \cos s\ A_1^+f(s)=C \cos s \int_0^{\cos s} f(t)(\cos^2s-t^2)^{-\frac {1}{2}}\ dt.
 $$
 Using the change of variable $t=r \cos s$ we have
\bea
 \cos s\ A_1^+f(s)&=&C \cos s \int_0^{1} f(r \cos s)\ (1-r^2)^{-\frac{1}{2}}\
 dr,\nonumber\\
 &\leq & C \cos s \int_0^{1} f(r \cos s)\ (1-r)^{-\frac {1}{2}}\
 dr.\label{csa} \eea
 For $f(s)= \chi_{E}(s)$ let, $$I(s)=\cos s\ A_1^+(\chi_{E})(s).$$ Using
 (\ref{csa}) we get
\bea I(s) &\asymp&C
\begin{cases} 0,\quad\mbox{if }
\cos s <a_1.\\
\cos s \left[\ds \sum_{i=1}^{j-1}\int_{a_i/\cos s}^{b_i/\cos s}\ (1-r)^{-\frac 12}\ dr +\int_{a_j/\cos s}^1\ (1-r)^{-\frac 12}\ dr\right],\\
\qquad \mbox{if } a_j\leq \cos s <b_j,  1\leq j \leq m. \\
\cos s \left[\ds \sum_{i=1}^{j}\int_{a_i/\cos s}^{b_i/\cos s}\ (1-r)^{-\frac 12}\ dr \right],\\
\qquad\mbox{if }
b_j\leq \cos s<a_{j+1},1\leq j \leq {m-1}.\\
\cos s \left[\ds \sum_{i=1}^{m}\int_{a_i/\cos s}^{b_i/\cos s}\ (1-r)^{-\frac 12}\ dr \right],\qquad\mbox{if } b_m < \cos s < 1.\end{cases} \\
&=&C
\begin{cases} 0,\quad\mbox{if}
\cos s <a_1.\\
\cos s \left[\ds \sum_{i=1}^{j-1}\left(1-\frac{a_i}{\cos
s}\right)^{\frac 12}-\left(1-\frac{b_i}{\cos s}\right)^{\frac 12}
+(1-\frac{a_j}{\cos s})^{\frac 12}\right],\\
\qquad \mbox{if }a_j\leq \cos s <b_j,  1\leq j \leq m. \\
\cos s \left[\ds \sum_{i=1}^{j}\left(1-\frac{a_i}{\cos s}\right)^{\frac 12}-\left(1-\frac{b_i}{\cos s}\right)^{\frac 12} \right],\label{caf}\\
\qquad\mbox{if }
b_j\leq \cos s <a_{j+1},1\leq j \leq {m-1}.\\
\cos s \left[\ds \sum_{i=1}^{m}\left(1-\frac{a_i}{\cos
s}\right)^{\frac 12}-\left(1-\frac{b_i}{\cos s}\right)^{\frac 12}
\right], \qquad\mbox{if } b_m < \cos s < 1.\end{cases}\eea By Lemma
\ref{lemma} (with $\gamma=2$) we have \bea I^2(s)&\leq&C
\begin{cases} 0,\quad\mbox{if }
\cos s <a_1.\\
\cos^2 s \left[\ds \sum_{i=1}^{j-1}\left(\frac{b_i-a_i}{\cos s}\right) +\left(\frac{\cos s-a_j}{\cos s}\right)\right],\\
\qquad \mbox{if }a_j \leq \cos s <b_j,  1\leq j \leq m. \\
\cos^2 s \left[\ds \sum_{i=1}^{j}\left(\frac{b_i-a_i}{\cos s}\right)
\right], \qquad\mbox{if }
b_j\leq \cos s<a_{j+1},1\leq j \leq {m-1}.\\
\cos^2 s \left[\ds \sum_{i=1}^{m}\left(\frac{b_i-a_i}{\cos s}\right)
\right], \qquad\mbox{if } b_m < \cos s < 1.\end{cases}\eea Thus we
have from above expression that \be I^2(s)\ \leq \ C\ \ds
\sum_{i=1}^{m}(b_i-a_i)  \ee We define $A_i=1-a_i,\ B_i=1-b_i,\
i=1,2,\ldots,m.$ Then $$A_1 \geq B_1\geq A_2\geq B_2 \cdots \geq A_m
\geq B_m.$$ Again using the lemma \ref{lemma} (with $\gamma=n/2$) we
get
\bea I^2(s)&\leq& \ C \ds \sum_{i=1}^{m}(b_i-a_i)\nonumber\\
&=&\ C \ds \sum_{i=1}^{m}(A_i-B_i)\nonumber\\
&\leq&\ C \left[\ds \sum_{i=1}^{m}\left(A_i^{\frac n2}-B_i^{\frac n2} \right)\right]^{2/n}\nonumber\\
&=&\ C \left[\ds \sum_{i=1}^{m}\left\{(1-a_i)^{\frac n2}-(1-b_i)^{\frac n2} \right\}\right]^{2/n}\nonumber\\
&=&\ C \left(\int_0^1\chi_E(s)(1-s)^{\frac{n-2}{2}}ds\right)^{2/n}\nonumber\\
&\asymp&\ C \left(\int_0^1\chi_E(s)(1-s^2)^{\frac{n-2}{2}}\
ds\right)^{2/n}.\nonumber \eea Using the change of variable $\cos t
=s,$ it follows that \be I(s)\leq C \left(\int_0^{\pi/2} \chi_E(\cos
t) \sin^{n-1} t dt\right)^{1/n}= C \|\chi_E\|_{L^{n,1}(S^n)}.\ee
This completes the proof.
\end{proof}
The following corollary can be thought of as an analogue of
(\ref{secondend}).
\begin{Corollary}
There exists a constant $C>0$ such that for all $f\in
L^{\frac{n}{d},1}(S^n),$

\be \|A^+_df\|_{L^{d+1,\infty}(G_{n+1,d+1})}\leq
C\|f\|_{L^{\frac{n}{d},1}(S^n)}\label{weirdeqsp}.\ee\label{weirdcorsp}
\end{Corollary}
\begin{proof}
We have from (\ref{endabelestsph}) that
$\|A^+_df\|_{L^{\infty}(G_{n+1,d+1})}\leq
\frac{C}{\cos\theta}\|f\|_{L^{\frac{n}{d},1}(S^n)}.$ Since the
function $g(\theta)=\frac{1}{\cos\theta}$ belongs to the space
$L^{d+1,\infty}[0,\frac{\pi}{2}]$ with respect to the measure
$\cos^{d}\theta\ d\theta$ the result follows.
\end{proof}
\begin{Remark}{\em
i) Interpolating between the estimates (\ref{endabelestsph}) and
(\ref{ppsph}) (for $p=1$) we get the following weighted estimate
which is somewhat analogous to $\R^n$:
$$\|\cos (\cdot )A_d^+f\|_{L^q(G_{n+1,d+1})}\leq
C\|f\|_{L^p(S^n)},$$ where
$$1\leq p<\frac{n}{d},\qquad \frac{n}{p}=\frac{n-d}{q}+d.$$

\noindent ii) Corollary \ref{weirdcorsp} has an interesting
implication. Suppose that $d+1>\frac{n}{d}$, for instance, we can
choose $d=n-1,$ $n>2.$ We can now use an interpolation argument
involving the estimates $\|A^+_df\|_{L^1(G_{n+1,d+1})}\leq
C\|f\|_{L^1(S^n)}$ and (\ref{weirdeqsp}). As a result we can prove that for each
$p\in (1,\frac{n}{d})$ there exist a $q>p$ such that,
$$\|A^+_df\|_{L^q(G_{n+1,d+1})}\leq C\|f\|_{L^p(S^n)}.$$ This implies that for radial functions
Theorem \ref{rubinsp} is not best possible.}
\end{Remark}
\begin{Example}{\em
 We now construct an example to show that (\ref{endabelestsph}) is not possible, if $1\leq p <
\frac{n}{d}.$ We consider the sequence $\{a_m\}$ where
$a_m=({m-1})/({m+1}).$ Then $\{a_m\}$ is an increasing sequence with
$\lim_{m\rightarrow\infty} a_m =1.$ We define a sequence of
functions $\{\tilde{f_m}\}$ by defining
$\tilde{f_m}(s)=\chi_{[a_m,1]}(s).$ We now consider the sequence of
radial functions on $S^n$ defined by
$f_m(\rho\sin\theta+x_0\cos\theta)=\tilde{f_m}(\cos\theta), \ 0 \leq
\theta \leq \pi/2.$ For $\pi/2 < \theta < \pi,$ we define $f_m$ by
the relation $f_m(x)=f_m(-x).$ We have
\be\|f_m\|_{L^{p,1}(S^n)}=C\left(\int_0^{\pi/2} \chi_{[a_m,1]}(\cos
s)\ \sin^{n-1} s\ ds \right)^{1/p}\label{fm} \ee By the change of
variable $\cos s=t$ in (\ref{fm}), we get \bea
\|f_m\|_{L^{p,1}(S^n)}&=& \ C \left(\int_0^{1} \chi_{[a_m,1]}(t)\ (1-t^2)^{\frac{n-2}{2}}\ dt \right)^{1/p}\nonumber\\
&\asymp &C\left(\int_0^{1} \chi_{[a_m,1](t)}\ (1-t)^{\frac{n-2}{2}}\ dt \right)^{1/p}\nonumber\\
&= &C(1-a_m)^{\frac{n}{2p}}\nonumber\\
&= &C ({m+1})^{-\frac{n}{2p}}.\label{fmlp} \eea If
$\frac{(m+1)a_m}{m}<\cos s < \frac{ma_m}{m-1},$ then it follows from
the explicit expression of the Abel transform (see (\ref{abelsph1}))
that \bea \cos s A_d^+f_m(s)&=& \frac{C}{\cos^{d-1} s} \int_0^{\cos
s}f_m(t)(\cos^2 s -t^2)^{\frac{d-2}{2}}\ dt\nonumber\\
&=& \frac{C}{\cos^{d-1} s} \int_{a_m}^{\cos s}(\cos^2 s
-t^2)^{\frac{d-2}{2}}\ dt.\label{exestsph} \eea By substituting $t=r
\cos s$ in (\ref{exestsph}) we get
 \bea
\cos s A_d^+f_m(s)&=& C \int_{\frac{a_m}{\cos
s}}^{1}(1-r^2)^{\frac{d-2}{2}}\ dr\nonumber\\
&\geq& C \left(1-\frac{a_m}{\cos s}\right)^{\frac d2}\nonumber\\
&\geq & C\left(1- \frac{m}{m+1}\right)^{\frac d2}\nonumber\\
&=& C(m+1)^{-\frac d2}.\label{adflp} \eea It follows from
(\ref{fmlp}) and (\ref{adflp}) that an estimate of the form
 $$\|\cos(\cdot)A_d^{+}f(\cdot)\|_{L^{\infty}(G_{n+1,d+1})}\leq C \|f\|_{L^{p,1}(S^n)}$$
is possible only if $\frac {n}{2p} \leq \frac{d}{2},$ that is, $p
\geq\frac{n}{d}.$}
\end{Example}


\begin{thebibliography}{99}

\bibitem {BR1} Carlos A. Berenstein and Boris Rubin,
{\em Radon transform of $L^p$-functions on the Lobachevsky space and
hyperbolic wavelet transforms.} Forum Math. no.  {\bf 5},  11
(1999),  567--590.

\bibitem {BR2} Carlos A. Berenstein and Boris Rubin, {\em Totally geodesic Radon transform of $L^p$-functions on real
hyperbolic space.} Fourier analysis and convexity,  37--58,  Appl.
Numer. Harmon. Anal., Birkhuser Boston, Boston, MA,  (2004).

\bibitem {Bray} William O. Bray, {\em Aspects of harmonic analysis on real
hyperbolic space.}  Fourier analysis (Orono, ME, 1992),  77--102.

\bibitem {CMS} Michael Cowling, Stefano Meda and Alberto G. Setti,  {\em An overview of harmonic analysis on
the group of isometries of a homogeneous tree.}  Exposition. Math.
{\bf 16}  (1998),  no. 5,  385--423.

\bibitem {DNO} Javier Duoandikoetxea, Virginia Naibo and Osane Oruetxebarria, {\em $k$-plane transforms and
related operators on radial functions.} Michigan Math. J.  {\bf 49}
(2001),  265--276.

\bibitem {Graf} L. Grafakos, {\em Classical and modern Fourier analysis},  (2004),  Pearson Education, Inc.
NJ.

\bibitem{Go} Fulton B. Gonzalez, {\em Radon transform on Grassmann manifolds.} Journal of Func. Anal., {bf 71}  (1987), 339--362.

\bibitem {GK1} Fulton B. Gonzalez and Tomoyuki Kakehi, {\em Pfaffian systems and Radon transforms on affine Grassmann manifolds.}
 Math. Ann., {\bf 326} (2003), no. 2, 237--273.

\bibitem {GK2} Fulton B. Gonzalez and Tomoyuki Kakehi, {\em Dual Radon transforms on affine Grassmann manifolds.}
Trans. Amer. Math. Soc. {\bf 356} (2004), no. 10,  4161--4180

\bibitem {He1} Sigurdur Helgason, {\em The Radon transform on Euclidean spaces, compact two-point homogeneous spaces and Grassmann manifolds}. Acta Math., {\bf 113} (1965), 153--180.

\bibitem {He2} Sigurdur Helgason, {\em The Radon transform}. Second edition. Progress in
Mathematics,  {\bf 5}.  Birkhuser Boston, Inc., Boston, MA, (1999).

\bibitem {He3} Sigurdur Helgason, {\em Groups and geometric analysis. Integral geometry, invariant differential operators,
and spherical functions.} Mathematical Surveys and Monographs,  {\bf
83}.  American Mathematical Society, Providence, RI,  (2000).

\bibitem {Io} Alexandru D. Ionescu, {\em An endpoint estimate for the Kunze-Stein phenomenon and related maximal operators.}
  Ann. of Math.  (2)  {\bf 152}  (2000),  no. 1,  259--275.

\bibitem {Is} Satoshi Ishikawa, {\em The range characterizations of the totally geodesic Radon transform on the real
  hyperbolic space.}  Duke Math. J.  {\bf 90}  (1997),  no. 1,  149--203.

\bibitem {KR2} Ashisha Kumar and Swagato K. Ray, {\em Weighted estimates of $d\mbox{-}$plane transform for radial functions on
Euclidean spaces.} Israel J. Math. {\bf 188}  (2012),  no. 1,  25--56.

\bibitem{M} Andrew Markoe, {\em Analytic tomography.} Encyclopedia of Mathematics and its Applications,  {\bf
106}.  Cambridge University Press, Cambridge,  (2006).

\bibitem {OS} D. M. Oberlin and E. M. Stein, {\em Mapping properties of the Radon transform.}
Indiana Univ. Math. J.  {\bf 31}  (1982),  641--650.

\bibitem {Rj} John G. Ratcliffe, {\em Foundations of hyperbolic manifolds.} Second edition. Graduate Texts in Mathematics,
  {\bf 149}.  Springer, New York,  (2006).

\bibitem {RS} Swagato K. Ray and Rudra P. Sarkar, {\em Fourier and Radon transform on harmonic $NA$ groups.}
Trans. Amer. Math. Soc.  {\bf 361}  (2009),  no. 8, 4269--4297.

\bibitem {R3} Boris Rubin, {\em Inversion formulas for the spherical Radon transform and the generalized cosine
transform.}  Adv. in Appl. Math.  {\bf 29}  (2002),  no. 3,
471--497.

\bibitem {R4} Boris Rubin, {\em Radon, cosine and sine transforms on real hyperbolic space.} Adv. Math. {\bf 170} (2002), no. 2, 206-–223.

\bibitem {R5} Boris Rubin, {\em Radon transforms on affine Grassmannians.}  Trans.
Amer. Math. Soc.  {\bf 356}  (2004), no. 12, 5045--5070.

\bibitem {S} Donald C. Solmon, {\em A note on $k$-plane integral transforms.}
J. Math. Anal. Appl.  {\bf 71}  (1979),  351--358.

\bibitem{Str1} Robert S. Strichartz,  {\em Harmonic analysis on Grassmannian bundles.} Trans. of the Amer. Math. Soc.,
 {\bf 296}  (1986), 387--409.

\bibitem {Str2} Robert S. Strichartz, {\em $L^p$ estimates for Radon transforms in Euclidean and non-Euclidean spaces.}
  Duke Math. J.  {\bf 48}  (1981),  no. 4,  699--727.

\bibitem {SW3} E. M. Stein and G. Weiss, {\em Introduction to Fourier analysis on Euclidean spaces.}
Princeton Mathematical Series, No.  {\bf 32.}  Princeton University
Press, Princeton, N.J.,  (1971).

\bibitem {VK} N. Ja. Vilenkin and A. U. Klimyk,
{\em Representation of Lie groups and special functions. Vol. 2.
Class I representations, special functions, and integral
transforms.}
 Translated from the Russian by V. A. Groza and A. A. Groza.
 Mathematics and its Applications (Soviet Series),  {\bf 74}.  Kluwer Academic Publishers Group, Dordrecht,  (1993).


\end{thebibliography}
\end{document}